\documentclass[11pt]{article}
\usepackage{amsmath,amsfonts,amsthm,amssymb, mathtools}
\usepackage{color}
\usepackage[colorlinks,bookmarksopen,bookmarksnumbered,citecolor=blue, linkcolor=red, urlcolor=blue]{hyperref}
\usepackage{authblk,needspace}
\hypersetup{
  pdftitle={Oriented trees in digraphs without short non-directed cycles},
  pdfauthor={Junying Lu, Yaojun Chen},
  pdfsubject={Oriented tree embeddings with no short non-directed cycles},
  pdfkeywords={digraph, oriented tree, girth, minimum semidegree}
}

\newtheorem{theorem}{Theorem}
\newtheorem{lemma}{Lemma}

\newtheorem{conjecture}{Conjecture}

\newtheorem{question}{Question}

\allowdisplaybreaks
\newcommand{\dist}{\operatorname{dist}}
\newcommand{\diam}{\operatorname{diam}}
\begin{document}

\title{\bf \Large  Oriented trees in digraphs without short non-directed cycles}

\author{Junying Lu$^1$}
\author{Yaojun Chen$^{2,}$\footnote{Corresponding author. yaojunc@nju.edu.cn.}}
\affil{\small $^1$School of Mathematics and Statistics, Nanjing University of Science and Technology,\\
Nanjing 210094, P.R. CHINA\\
$^2$School of Mathematics, Nanjing University, Nanjing 210093, P.R. CHINA}
\date{ }

\maketitle

\begin{abstract}
The girth of a graph $G$ is the length of a shortest cycle of $G$.
Jiang (JCT-B, 2001) showed that every graph $G$ with girth at least
$2\ell+1$ and minimum degree at least $k/\ell$ contains every tree
with $k$ edges whose maximum degree does not exceed the minimum
degree of $G$. In this paper, we extend Jiang's result to digraphs
by proving that every digraph $D$ with no non-directed cycle of
length between $3$ and $2\ell$ and minimum semidegree at least
$k/\ell$ contains every oriented tree with $k$ edges whose maximum
degree does not exceed the minimum semidegree of $D$.
This answers a question raised by Stein and Trujillo-Negrete
in the affirmative.\\

\end{abstract}

\baselineskip=0.202in

\section{Introduction}
A typical problem in extremal graph theory is the existence of a specific subgraph under the given degree condition. A fundamental result states that any graph $G$ with minimum degree $\delta(G)\ge k$ contains all trees with $k$ edges, provable by a simple greedy embedding. Although this bound on the minimum degree is generally the best possible, there is room for improvement if an additional condition is imposed on $G$, such as the requirements on the maximum degree, denoted by $\Delta(G)$, and the girth, defined as the length of the shortest cycle in $G$, and so on. On the other hand, considering the average degree, Erd\H{o}s and Sós \cite{Erdos} proposed the following conjecture.
\begin{conjecture}[Erd\H{o}s and S\'os \cite{Erdos}]\label{conj-ES}
 Every graph of order $n$ with more than $n(k-1)/2$ edges (or, equivalently, with average degree greater than $k-1$) contains every tree with $k$ edges as a subgraph.
\end{conjecture}
Brandt and Dobson \cite{Brandt} showed that every graph of girth at least $5$, with $\delta(G)\ge k/2$ and $\Delta(G)\ge \Delta(T)$, contains each tree $T$ with $k$ edges. 
\begin{theorem}[Brandt and Dobson \cite{Brandt}]
Let $G$ be a graph of girth at least $5$ and $T$ be a tree with $k$ edges. If $\delta(G)\ge k/2$ and $\Delta(G)\ge \Delta(T)$, then $G$ contains $T$.     
\end{theorem}
It follows that Conjecture \ref{conj-ES} holds for graphs with girth at least $5$. Generalizing this result, Sacl\'{e} and Wo\'{z}niak \cite{Sacle} proved Conjecture \ref{conj-ES} for graphs without 4-cycles. Actually, they showed a slightly stronger result.

\begin{theorem}[Sacl\'{e} and Wo\'{z}niak \cite{Sacle}]\label{thm-SW}
Let $G$ be a graph containing no $4$-cycle and $T$ be a tree with $k$ edges. If $\delta(G)\ge k/2$ and $\Delta(G)>\Delta(T)$ (or $\Delta(G)\ge k$ if $T$ is a star), then $G$ contains $T$.   
\end{theorem}
 
Dobson \cite{Dobson} further conjectured that any graph of girth at least $2\ell+1$ with $\delta(G)\ge k/\ell$ and $\delta(G)\ge \Delta(T)$ contains each tree $T$ with $k$ edges. This conjecture was subsequently confirmed by Jiang \cite{Jiang} after prior work by Haxell and \L uczak \cite{Haxell}.
\begin{theorem}[Jiang \cite{Jiang}]\label{thm-J}
Let $G$ be a graph with girth at least $2\ell+1$ and $T$ be a tree with at most $k$ edges. If $\delta(G)\ge \max\{k/\ell,\Delta(T)\}$, then $G$ contains $T$ as a subgraph. 
\end{theorem}
Note that the condition $\delta(G)\ge \Delta(T)$ (instead of $\Delta(G)\ge \Delta(T)$) is indeed necessary when $\ell\ge 3$, which can be seen by considering a balanced double star and a suitable host graph (see \cite{Jiang} for details).

It would be interesting to find extensions of these results to digraphs. All graphs and digraphs considered in this paper are finite. Digraphs have no loops or parallel arcs, but may contain both $(x,y)$ and $(y,x)$. One has to decide which parameter will play the role of the minimum degree. A natural choice is the minimum semidegree $\delta^0(D)$, a widely used parameter defined as the minimum of the in-degrees and out-degrees of the vertices of $D$. In the same way as in the undirected case, one can use a greedy embedding strategy to see that any digraph $D$ with $\delta^0(D)\ge k$ must contain each oriented tree with $k$ edges. And as before, it seems reasonable to ask whether this bound can be lowered if the underlying graph $G$ of $D$ (i.e., the graph we obtain by omitting directions) satisfies some extra conditions.

For a digraph $D$, the maximum degree $\Delta(D)$ is the maximum degree of its underlying graph. Denote the maximum out-degree and in-degree by $\Delta^+(D)$ and $\Delta^-(D)$, respectively. Write $\Delta^\pm(D)=\min\{\Delta^+(D),\Delta^-(D)\}$. 

Stein and Trujillo-Negrete \cite{Stein} proved the following result, which is a digraph version of Theorem \ref{thm-SW}.

\begin{theorem}[Stein and Trujillo-Negrete \cite{Stein}]
Let $T$ be an oriented tree with $k$ edges, and let $D$ be a digraph with $\delta^0(D)\ge k/2$ having no oriented $4$-cycles. If $\Delta^\pm(D)>\Delta(T)$, then $D$ contains $T$ as a subgraph. 
\end{theorem}
Note that an \emph{oriented graph} is a digraph that contains no directed cycles of length two. Stein and Trujillo-Negrete \cite{Stein} also asked whether there is a natural analogue of Theorem \ref{thm-J}.
\begin{question}[Stein and Trujillo-Negrete \cite{Stein}]\label{q}
Is it true that every oriented graph $D$ of girth at least $2\ell+1$ with $\delta^0(D)\ge \max\{k/\ell,\Delta(T)\}$ contains every oriented tree $T$ with $k$ edges?
\end{question}

To state our results, we first introduce some notation. For a family $\mathcal H$ of digraphs, a digraph is
$\mathcal H$-free if it contains no member of $\mathcal H$
as a subgraph. Let 
$$\mathcal{C}_{\le 2\ell}=\{C \mid C \text{ is an oriented cycle and } 3\le |C|\le 2\ell\},$$ 
and 
$$\mathcal{C}_{\le 2\ell}^*=\mathcal{C}_{\le 2\ell}\setminus \{\text{all directed cycles}\}.$$ 
Note that, for an oriented graph, having girth at least $2\ell+1$ is equivalent to being $\mathcal C_{\le2\ell}$-free, but a $\mathcal{C}_{\le 2\ell}$-free or $\mathcal{C}_{\le 2\ell}^*$-free digraph may not be oriented. We answer Question \ref{q} under the weaker assumption that the host digraph is $\mathcal{C}_{\le 2\ell}^*$-free.

\begin{theorem}\label{thm-2}
Let $\ell\ge1$ be an integer, $T$ be an oriented tree with $k\ge1$ edges, and $D$ be a digraph with $\delta^0(D)\ge\max\{k/\ell,\Delta(T)\}$. If $D$ is $\mathcal{C}_{\le 2\ell}^*$-free, then $D$ contains $T$.
\end{theorem}

The condition $\delta^0(D)\ge\Delta(T)$ cannot be omitted. Erd\H{o}s and Sachs \cite{Sachs} proved the existence of regular graphs of any prescribed degree at least $2$ and arbitrarily large girth. Let $T$ be an oriented tree with $\Delta^+(T)=\Delta(T)\ge2$, and let $G$ be a $2(\Delta(T)-1)$-regular graph of girth at least $2\ell+1$. Orienting each component of $G$ along an Euler tour yields an orientation $D$ with in-degree and out-degree $\Delta(T)-1$ at every vertex. Clearly, $D$ does not contain $T$. Taking $T$ to be an out-star also ensures that $\delta^0(D)\ge k/\ell$ whenever $\ell\ge2$.

An oriented tree is \emph{antidirected} if it contains no directed path of length $2$. If the tree we are looking for is antidirected, we can replace the minimum semidegree with the \emph{minimum pseudo-semidegree} of $D$, denoted by $\Bar{\delta}^0(D)$. This is $0$ if $D$ has no arcs, and otherwise is defined as the smaller of the minimum positive out-degree and the minimum positive in-degree. Clearly, $\delta^0(D) \leq \Bar{\delta}^0(D)$. For an antidirected tree,  we can slightly strengthen Theorem \ref{thm-2} as below.

\begin{theorem}\label{thm-3}
Let $\ell\ge1$ be an integer, $T$ be an antidirected tree with $k\ge1$ edges, and $D$ be a digraph with $\bar\delta^0(D)\ge\max\{k/\ell,\Delta(T)\}$. If $D$ is $\mathcal{C}_{\le 2\ell}^*$-free, then $D$ contains $T$.
\end{theorem}

Both results imply Theorem \ref{thm-J}. Indeed, let $T$ be a tree with $k$ edges and $G$ be a graph of girth at least $2\ell+1$ with
$\delta(G)\ge\max\{k/\ell,\Delta(T)\}$.
Consider the digraph $D$ obtained by replacing each edge
of $G$ with a directed $2$-cycle.
Then $D$ is $\mathcal C_{\le2\ell}^*$-free and satisfies
$\delta^0(D)=\bar\delta^0(D)=\delta(G)$.
Applying Theorem \ref{thm-2} to an arbitrary orientation of $T$,
or Theorem \ref{thm-3} to an antidirected orientation of $T$,
we conclude that $D$ contains such an orientation of $T$ and hence
$G$ contains $T$.

\section{Preliminaries}
In this section, we first introduce some additional terminology and notations that will be used throughout this paper, and then give some lemmas for proving our main results.

Let $D$ be a digraph. When we write an edge $xy$, we do not
specify its orientation, that is, $(x,y)\in E(D)$ or
$(y,x)\in E(D)$. For any set $U\subseteq V(D)$, we denote by
$D[U]$ the subgraph of $D$ induced by $U$, and write
$D-U=D[V(D)\setminus U]$. In particular, when $U$ contains
a single vertex, say $u$, we simply write $D-u$ instead of
$D-\{u\}$.

For $v\in V(D)$, the out-neighborhood (in-neighborhood) of $v$,
denoted by $N_D^+(v)$ ($N_D^-(v)$), is the set of out-neighbors
(in-neighbors) of $v$. Let $N_D(v)=N_D^+(v)\cup N_D^-(v)$.
The out-degree, in-degree and degree of $v$ are
$d_D^+(v)=|N_D^+(v)|$, $d_D^-(v)=|N_D^-(v)|$ and
$d_D(v)=|N_D(v)|$, respectively. Thus, the degree refers to
the underlying graph. The \emph{distance}
$\operatorname{dist}_D(u,v)$ of vertices $u,v\in V(D)$ is
the length of a shortest path of any orientation that joins
them, and is infinite if no such path exists.
The \emph{diameter} of $D$, denoted by $\operatorname{diam}(D)$,
is the maximum distance between any pair of its vertices.

Let $T$ be an oriented tree. The \emph{leaves} of $T$ are
the vertices having exactly one neighbor. There are two kinds
of leaves, namely \emph{in-leaves}, which have out-degree $1$
and in-degree $0$, and \emph{out-leaves}, which have out-degree
$0$ and in-degree $1$. Two leaves with the same neighbor in $T$
are \emph{siblings}. The \emph{skeleton} of $T$, denoted by
$S(T)$, is the subgraph obtained by deleting all leaves of $T$.
A \emph{penultimate vertex} in $T$ is a leaf in $S(T)$.
A component in a graph is \emph{nontrivial} if it contains
at least two vertices. For $x,y\in V(T)$, $T[x,y]$ denotes
the unique path in $T$ joining $x$ and $y$.

Given an oriented tree $T$ and a digraph $D$, an \emph{embedding}
of $T$ in $D$ is an injective function $f\colon V(T)\to V(D)$
such that $(f(u),f(v))\in E(D)$ whenever $(u,v)\in E(T)$.
If such an embedding exists, we say that $T$ embeds in $D$,
or $T$ is a subgraph of $D$. For a subgraph $H$ of $T$ and an embedding $f$ from $T$ to $D$, we let $f(H)$ denote the image of $H$ in $D$, or the vertex set of $f(H)$.

\begin{lemma}[Jiang \cite{Jiang}]\label{lem-1}
Let $G$ be a connected graph of order $n$ and $X\subseteq V(G)$. If every pair of distinct vertices in $X$ has distance at least $2\ell-1$ in $G$, then $|X|\le\max\{\lfloor n/\ell\rfloor,1\}$.
\end{lemma}

\begin{lemma}[Jiang \cite{Jiang}]\label{lem-2}
A tree with $t\ge2$ leaves has at most $t-2$ vertices of degree at least $3$.
\end{lemma}

We will also use the following estimate for pairs of vertices
at large distance in a tree.

\begin{lemma}\label{lem-pairs}
Let $T$ be a tree on $n$ vertices, and let $h\ge1$ be an integer.
The number of vertex pairs whose distance is at least
$h$ is zero if $n<h$ and at most $\binom{n-h+1}{2}$ if $n\ge h$.
Moreover, if $h\ge3$, $n\ge h+3$, and two distinct vertices of
$T$ are each adjacent to at least two leaves, then this number is
at most $\binom{n-h+1}{2}-2$.
\end{lemma}
\begin{proof}
We prove the first assertion by induction on $n$. It is immediate for $n\le h$. Suppose that $n>h$, and let $v$ be a leaf. We claim that there are at least $h$ vertices, including $v$, at distance less than $h$ from $v$. Indeed, a path from $v$ of
length at least $h-1$ contains $h$ such vertices, while if no
such path exists, every vertex is at distance less than $h$
from $v$. Thus, at most $n-h$ vertices are at distance at least
$h$ from $v$. Deleting the leaf $v$ preserves the distance between every pair of vertices in $T-v$, so by the induction hypothesis, the number of such pairs is at most $\binom{n-h}{2}+n-h=\binom{n-h+1}{2}$.

Suppose now that $h\ge3$ and $v$ has a sibling. If there exists a vertex at distance at least $h$ from $v$,
then the path from $v$ to that vertex contains $h$ vertices
at distance less than $h$ from $v$. Since $h\ge3$ and $v$
has a sibling, there are at least $h+1$ vertices at distance
less than $h$ from $v$. Thus, for $n\ge h+1$, at most $n-h-1$ vertices have distance at least $h$ from $v$. Let $u_1$ and $u_2$ be two distinct vertices, each adjacent
to at least two leaves. For each $i\in\{1,2\}$, choose a leaf
$v_i$ adjacent to $u_i$. Since $v_1$ has a sibling, there are
at most $n-h-1$ pairs at distance at least $h$ containing
$v_1$. The tree $T-v_1$ has $n-1$ vertices, and $v_2$ still
has a sibling in this tree. Hence, there are at most $n-h-2$
pairs at distance at least $h$ containing $v_2$ but not $v_1$.
The tree $T-\{v_1,v_2\}$ has $n-2$ vertices. By the first
assertion, there are at most $\binom{n-h-1}{2}$ pairs at
distance at least $h$ containing neither $v_1$ nor $v_2$.
Therefore, the total number of such pairs is at most
\[
\binom{n-h-1}{2}+(n-h-1)+(n-h-2)
=\binom{n-h+1}{2}-2.
\]
\end{proof}

\begin{lemma}\label{lem-distance}
Let $\ell\ge2$ be an integer and $D$ be a
$\mathcal{C}_{\le 2\ell}^*$-free digraph.
Suppose that $x\in V(D)$ and $H$ is an oriented tree in $D-x$.
Then, for each $\diamond\in\{+,-\}$, any two distinct vertices of
$N_D^\diamond(x)\cap V(H)$ are at distance at least
$2\ell-1$ in $H$.
\end{lemma}
\begin{proof}
Suppose that $a,b\in N_D^\diamond(x)\cap V(H)$ are distinct and $\dist_H(a,b)\le2\ell-2$. The path $H[a,b]$ and the two arcs joining $x$ to $a,b$ form an oriented cycle $C$, and $3\le |C|\le 2\ell$.
Since $d_C^\diamond(x)=2$, $C$ is non-directed, a contradiction.
\end{proof}

\begin{lemma}\label{lem-path}
Let $\ell\ge2$ be an integer and $D$ be a
$\mathcal{C}_{\le 2\ell}^*$-free digraph with distinct
vertices $x,y$.
For any oriented path $P$ of length at most $\ell$ with
endvertices $a,b$, there is at most one embedding $f$ of $P$
in $D$ such that $f(a)=x$ and $f(b)=y$.
Moreover, there is at most one directed path from $x$ to $y$
of length at most $\ell$.
\end{lemma}
\begin{proof}
Suppose that $f,g$ are distinct embeddings of $P$ in $D$
with $f(a)=g(a)=x$ and $f(b)=g(b)=y$.
Write $Q_1=f(P)$ and $Q_2=g(P)$.
Following $P$ from $a$ to $b$, let $v$ be the first vertex
with $f(v)\ne g(v)$, and denote its preceding vertex by $u$.
Write $z=f(u)=g(u)$.
The arcs joining $z$ to $f(v)$ and $g(v)$ are images
of the same arc of $P$.
Let $w$ be the first vertex after $z$ on $Q_1$ that belongs
to $Q_2[z,y]$. Such a vertex exists because both paths end
at $y$. By the choice of $w$, the subpaths $Q_1[z,w]$ and
$Q_2[z,w]$ are internally disjoint. Their union forms an
oriented cycle $C$ of length between $3$ and $2\ell$.
Since $d_C^+(z)=2$ or $d_C^-(z)=2$, $C$ is non-directed,
a contradiction.

For the second assertion, suppose that there are two distinct
directed paths from $x$ to $y$, each of length at most $\ell$.
The same construction as above produces a cycle of length between
$3$ and $2\ell$ whose two arcs incident with $z$ both leave $z$,
again a contradiction.
\end{proof}

The following is a basic embedding lemma that plays an important role in proving Theorems \ref{thm-2} and \ref{thm-3}.
\begin{lemma}\label{lem-E}
For integers $\ell,d\ge2$, suppose that $D$ is a
$\mathcal{C}_{\le 2\ell}^*$-free digraph and $T$ is an oriented
tree with $\Delta(T)\le d$ and diameter at most $2\ell-2$.
If $\delta^0(D)\ge d$, or if $T$ is antidirected and
$\bar\delta^0(D)\ge d$, then $D$ contains $T$.
\end{lemma}
\begin{proof}
Suppose $T$ has $k$ edges. We prove the result by induction on $k$.
The result holds trivially when $k=1$. Assume now that $k\ge2$.
Let $w$ be a leaf of $T$. By symmetry, we may assume that $w$
is an out-leaf. Let $w'$ be its in-neighbor in $T$.
By the induction hypothesis, we have an embedding $f$ from
$T-w$ to $D$. Since $\diam(T)\le2\ell-2$, Lemma \ref{lem-distance}
implies that each component of $f(T-w-w')$ contains at most
one out-neighbor of $f(w')$. As $T-w-w'$ has $d_T(w')-1$
components, at most $d_T(w')-1$ out-neighbors of $f(w')$
lie in $f(T-w)$. Since $f(w')$ has an out-neighbor in $f(T-w)$ when $T$ is antidirected, we have
$d_D^+(f(w'))\ge d\ge\Delta(T)\ge d_T(w')$, whether or not $T$ is antidirected. It follows that there is an out-neighbor of $f(w')$ outside $f(T-w)$, and then we can extend $f$ to an embedding of $T$ in $D$. 
\end{proof}
The same induction also shows that any prescribed arc of $T$
can be mapped to any prescribed arc of $D$, although this
stronger conclusion is not needed below.

\smallskip

Fix $\ell,d\ge2$ and a $\mathcal{C}_{\le 2\ell}^*$-free digraph $D$. We consider oriented trees when $\delta^0(D)\ge d$, and antidirected trees when $\bar\delta^0(D)\ge d$. In the corresponding class, a \emph{minimal nonembeddable tree} is a tree $T$ with at most $\ell d$ edges and maximum degree at most $d$ that does not embed in $D$, but every proper subtree does. By Lemma \ref{lem-E}, such a tree is not a star, so its skeleton has at least two leaves.

\begin{lemma}\label{lem-p1}
A minimal nonembeddable tree has at most $\ell-1$ penultimate vertices.
\end{lemma}
\begin{proof}
Let $T$ be a minimal nonembeddable tree with $k$ edges.
Suppose to the contrary that $T$ contains at least $\ell$
penultimate vertices. Let $u$ be a penultimate vertex of $T$
such that $d_T(u)$ is as small as possible, and let $d_u^+$
and $d_u^-$ be the numbers of leaves of $T$ in $N_T^+(u)$
and $N_T^-(u)$, respectively. Let $v_1,\ldots,v_t$ be the
leaf neighbors of $u$. Clearly, $d_u^++d_u^-=t$ and
$1\le t\le d-1$. Let $u_1,\ldots,u_{\ell-1}$ denote
$\ell-1$ penultimate vertices of $T$ other than $u$.
By the choice of $u$, each $u_i$ is adjacent to at least
$t$ leaves. By the minimality of $T$, there exists an embedding $f$ of
$T'=T-\{v_1,\ldots,v_t\}$ in $D$. Let $T''=T'-u$.
Then $f(T'')$ is a tree on $k-t$ vertices.
Let
$S^+=N_D^+(f(u))\cap f(T'')$ and
$S^-=N_D^-(f(u))\cap f(T'')$.
By Lemma \ref{lem-distance}, the distance in $f(T'')$
between every pair of distinct vertices in $S^+$ is at least
$2\ell-1\ge3$. In particular, no two sibling leaves in
$f(T'')$ belong to $S^+$. Hence, for each
$i\in\{1,\ldots,\ell-1\}$, at most one leaf adjacent to
$f(u_i)$ belongs to $S^+$. We can therefore delete $t-1$
leaves adjacent to each $f(u_i)$ from $f(T'')$ without
deleting any member of $S^+$. Denote the resulting subtree
by $T^+$. Then $|T^+|=k-\ell t+\ell-1$, and the distance
between every pair of vertices in $S^+$ remains unchanged.
Applying Lemma \ref{lem-1} to $S^+$ in $T^+$, we obtain
\[
|S^+|
 \le \max\left\{\left\lfloor\frac{|T^+|}{\ell}\right\rfloor,1\right\}
 =\max\left\{\left\lfloor
 \frac{k-\ell t+\ell-1}{\ell}\right\rfloor,1\right\}
 \le d-t,
\]
where the last inequality follows from $k\le\ell d$ and
$d-t\ge1$. By the same argument, $|S^-|\le d-t$.

If $\delta^0(D)\ge d$, then $f(u)$ has at least $t$
out-neighbors and at least $t$ in-neighbors outside
$f(T'')$. If $T$ is antidirected and
$\bar\delta^0(D)\ge d$, then $|N_D^\diamond(f(u))\setminus S^\diamond|\ge \max\{0,d_u^\diamond\}$ for each $\diamond\in\{+,-\}$. In both cases, there exist
two disjoint subsets
$L^+\subseteq N_D^+(f(u))\setminus f(T'')$ and
$L^-\subseteq N_D^-(f(u))\setminus f(T'')$
such that $|L^+|\ge d_u^+$ and $|L^-|\ge d_u^-$.
Therefore, we can extend $f$ to an embedding of $T$ in $D$
by embedding $v_1,\ldots,v_t$ in $L^+\cup L^-$,
a contradiction.
\end{proof}

\begin{lemma}\label{lem-p2}
Let $T$ be a minimal nonembeddable tree.
If $u\in V(T)$ is adjacent to exactly one leaf, then $T-u$
has a nontrivial component with fewer than $\ell$ vertices.
Moreover, every penultimate vertex of $T$ is adjacent to
at least two leaves.
\end{lemma}
\begin{proof}
Write $k=|E(T)|$, and suppose that $v$ is the unique leaf adjacent
to $u$. Let $T'=T-v$. By the minimality of $T$, there exists
an embedding $f$ of $T'$ in $D$. Let $T_1,\ldots,T_m$ be the
components of $T'-u$, and write $|T_i|=n_i$ for $i\in[m]$.
These components are nontrivial, since $v$ is the unique leaf
adjacent to $u$. Let $\diamond\in\{+,-\}$ be such that
$v\in N_T^\diamond(u)$, and let
$S_i^\diamond=N_D^\diamond(f(u))\cap f(T_i)$.
If $n_i\ge\ell$ for every $i\in[m]$, then Lemmas
\ref{lem-1} and \ref{lem-distance} imply
\[
|N_D^\diamond(f(u))\cap f(T'-u)|
 =\sum_{i=1}^m|S_i^\diamond|
 \le\sum_{i=1}^m\left\lfloor\frac{n_i}{\ell}\right\rfloor
 \le\left\lfloor\frac{k-1}{\ell}\right\rfloor
 \le d-1.
\]
Since $d_D^\diamond(f(u))\ge d$, we can therefore extend $f$ to an embedding of $T$ in $D$ by mapping $v$ to a vertex in $N_D^\diamond(f(u))\setminus f(T'-u)$, a contradiction. This proves the first part of the conclusion.

Now, let $u$ be a penultimate vertex of $T$.
By Lemma \ref{lem-E}, we may assume that $T$ has diameter
at least $2\ell-1$. The graph $T-u$ consists of some isolated
vertices and one nontrivial component. This component has at
least $\ell$ vertices, since otherwise $T$ would have diameter
at most $\ell$. By the first part of the conclusion, $u$ is
adjacent to either no leaves or at least two leaves.
Since $u$ is adjacent to at least one leaf, it must be adjacent
to at least two leaves.
\end{proof}

Root a tree $T$ at $r$. A vertex is an ancestor of another if it lies on the path from $r$ to that vertex. Two vertices are \emph{comparable} if one is an ancestor of the other, and are \emph{incomparable} otherwise. Define
\begin{equation*}
\begin{split}
\mathcal E_\ell(T,r)={}&
 \bigl\{\{a,b\}\subseteq V(T)\mid \dist_T(a,b)\ge2\ell+1\bigr\}\\
&{}\cup\bigl\{\{a,b\}\subseteq V(T)\mid a,b\text{ are incomparable},\
 \ell+2\le\dist_T(a,b)\le2\ell\bigr\}.
\end{split}
\end{equation*}

The following lemma provides a sufficient condition for $T$
to embed in $D$ in terms of $|\mathcal E_\ell(T,r)|$.
We will use it to analyze a minimal nonembeddable tree.
\begin{lemma}\label{lem-random}
For integers $\ell,d\ge3$, let $D$ be a
$\mathcal{C}_{\le 2\ell}^*$-free digraph and $T$ an oriented
tree with $\Delta(T)\le d$. Suppose that $\delta^0(D)\ge d$,
or that $T$ is antidirected and $\bar\delta^0(D)\ge d$.
If there is a root $r$ of $T$ such that
$|\mathcal E_\ell(T,r)|<(d-1)^\ell$, then $D$ contains $T$.
\end{lemma}
\begin{proof}
For $x\in V(D)$, let $B_x$ be the bipartite graph whose two
parts are disjoint copies of $N_D^+(x)$ and $N_D^-(x)$.
We denote the copies of $a\in N_D^+(x)$ and
$b\in N_D^-(x)$ by $a^+$ and $b^-$, respectively.
Join $a^+$ and $b^-$ in $B_x$ if there is a directed cycle
$xa\cdots bx$ of length between $2$ and $\ell+1$. In particular, if $x$ and $a$ are joined by a bidirected edge in $D$, then $a^+a^-\in E(B_x)$.
We claim that $B_x$ is a matching. Indeed, suppose that
$a^+b_1^-$ and $a^+b_2^-$ are edges of $B_x$.
Deleting $(x,a)$ from the corresponding cycles gives two
directed paths from $a$ to $x$ of length at most $\ell$.
By Lemma \ref{lem-path}, these paths are identical, and hence
$b_1=b_2$. Similarly, if $a_1^+b^-$ and $a_2^+b^-$ are edges
of $B_x$, then $a_1=a_2$.

Fix an order of the children of every vertex of $T$. We construct a random map $f\colon V(T)\to V(D)$ that preserves arcs. Throughout the construction, we maintain that any two distinct
comparable vertices whose distance in $T$ is at most $2\ell$
have distinct images. Under the condition $\delta^0(D)\ge d$, choose $f(r)$ arbitrarily. Under the condition that $T$ is antidirected and $\bar\delta^0(D)\ge d$, let $\diamond\in\{+,-\}$ be such that every child of $r$ belongs to $N_T^\diamond(r)$, and choose $f(r)$ with $d_D^\diamond(f(r))>0$.
Consequently,
$d_D^\diamond(f(r))\ge d$ for every
$\diamond\in\{+,-\}$ such that
$N_T^\diamond(r)\ne\emptyset$.

Suppose that $f(u)=x$ has been determined and $u\ne r$.
Let $A_u$ be the images of the nearest $2\ell-1$ strict
ancestors of $u$, or of all its strict ancestors if there are
fewer. These ancestors are consecutive on the $r,u$-path.
By the maintained property, their images are pairwise distinct
and different from $x$. Since $f$ preserves arcs, they form a
simple path in $D-x$ on at most $2\ell-1$ vertices. By Lemma \ref{lem-distance},
$|A_u\cap N_D^\diamond(x)|\le1$ for each
$\diamond\in\{+,-\}$. For every $\diamond\in\{+,-\}$ such that
$N_T^\diamond(u)$ contains a child of $u$, we have
$d_D^\diamond(x)\ge d$. In the antidirected case, this follows from
$\bar\delta^0(D)\ge d$, since the image of the parent of $u$
belongs to $N_D^\diamond(x)$. Therefore, for each such $\diamond$, we have $|N_D^\diamond(x)\setminus A_u|\ge d-1$,
so we can choose $d-1$ vertices from this set.
List them as $a_1,\ldots,a_{d-1}$ when $\diamond=+$,
and as $b_1,\ldots,b_{d-1}$ when $\diamond=-$.
If both lists are defined, order them so that
$a_p^+b_q^-\in E(B_x)$ implies $p=q$. This is possible because $B_x$ is a matching.
In particular, if $a_p=b_q$, then $p=q$.

Choose $\pi_u$ uniformly at random from all permutations of
$\{1,\ldots,d-1\}$. If the children of $u$ are
$v_1,\ldots,v_s$, then $s=d_T(u)-1\le\Delta(T)-1\le d-1$.
Map $v_i$ to $a_{\pi_u(i)}$ if $(u,v_i)\in E(T)$, and to
$b_{\pi_u(i)}$ if $(v_i,u)\in E(T)$.
At the root $r$, for each $\diamond\in\{+,-\}$ with
$N_T^\diamond(r)\ne\emptyset$, choose $d$ vertices from
$N_D^\diamond(f(r))$, order the two lists as above, and use
a uniform random permutation of $\{1,\ldots,d\}$.
The images of the children are distinct, and no two children
are mapped to the endpoints of an edge of $B_x$. Indeed, if
$f(v_i)=f(v_j)$, or if
$f(v_i)=a_{\pi_u(i)}$, $f(v_j)=b_{\pi_u(j)}$, and
$a_{\pi_u(i)}^+b_{\pi_u(j)}^-\in E(B_x)$, then the choice and
ordering of the lists imply $\pi_u(i)=\pi_u(j)$, and hence
$i=j$. By construction, each child also has an image different
from those of all its ancestors at distance at most $2\ell$
in $T$. Thus, the maintained property is preserved throughout
the construction.

In the above construction, the lists at each vertex $u$
and their order are chosen by a fixed rule depending only
on the images of the vertices on $T[r,u]$.
The random permutations at distinct vertices are chosen
independently. Now let $u\ne r$, and condition on all
permutations associated with vertices outside the subtree
rooted at $u$. Then $f(u)$ and the lists at $u$ are fixed,
while $\pi_u$ is still uniformly distributed.
Thus, each child of $u$ is mapped to any prescribed vertex
of $D$ with conditional probability at most $1/(d-1)$.

For distinct $a,b\in V(T)$, let $B_{a,b}$ be the event that $f(a)=f(b)$ and the images on $T[a,b]$ have no other repetition.  Suppose that $f$ is not injective, and choose distinct vertices
$a,b\in V(T)$ with $f(a)=f(b)$ such that $\dist_T(a,b)$ is as
small as possible. Then $B_{a,b}$ occurs. By the construction,
$\dist_T(a,b)\ge3$, and the image of $T[a,b]$ is a simple
oriented cycle $C$. Suppose that $\dist_T(a,b)\le2\ell$.
Since $D$ is $\mathcal C_{\le2\ell}^{*}$-free, $C$ is directed.
By the maintained property, $a$ and $b$ are incomparable.
Let $w$ be the last common ancestor of $a$ and $b$.
If $\dist_T(a,b)\le\ell+1$, then the two children of $w$
on $T[a,b]$ are mapped to the endpoints of an edge of
$B_{f(w)}$, contrary to the construction. Consequently,
either $\dist_T(a,b)\ge2\ell+1$, or $a$ and $b$ are
incomparable and
$\ell+2\le\dist_T(a,b)\le2\ell$.
Thus, $\{a,b\}\in\mathcal E_\ell(T,r)$.

We claim that $\mathbb P(B_{a,b})\le(d-1)^{-\ell}$ for every
$\{a,b\}\in\mathcal E_\ell(T,r)$. First suppose that $a$ and
$b$ are comparable. By symmetry, assume that $a$ is an ancestor
of $b$. Since $\{a,b\}\in\mathcal E_\ell(T,r)$, we have
$\dist_T(a,b)\ge2\ell+1$. Let $z$ be the ancestor of $b$ at
distance $\ell$ from $b$,
and let $\mathcal F_z$ denote the random tuple of all permutations
associated with vertices outside the subtree rooted at $z$.
Conditioning on $\mathcal F_z$ fixes $f(a)$ and $f(z)$,
while all permutations used along $T[z,b]$ remain unexposed.
On $B_{a,b}$, we have $f(z)\ne f(a)=f(b)$, and the image of
$T[z,b]$ is a simple oriented path of length $\ell$ from
$f(z)$ to $f(a)$.
By Lemma \ref{lem-path}, there is at most one embedding
of $T[z,b]$ with these endpoint images.
Thus, all $\ell$ choices along $T[z,b]$ must take specified
values. Revealing these choices successively, we obtain
\(
\mathbb P(B_{a,b}\mid\mathcal F_z)\le(d-1)^{-\ell}.
\)
Therefore, by the law of total probability,
\[
\mathbb P(B_{a,b})
=\sum_{\varphi}
\mathbb P(B_{a,b}\mid\mathcal F_z=\varphi)
\mathbb P(\mathcal F_z=\varphi)
\le(d-1)^{-\ell},
\]
where $\varphi$ ranges over all possible realizations of the
permutations in $\mathcal F_z$.
Now suppose that $a$ and $b$ are incomparable, and let $w$ be
their last common ancestor. Since
$\{a,b\}\in\mathcal E_\ell(T,r)$, we have
$\dist_T(a,b)\ge\ell+2$. Hence, there are nonnegative integers
$i,j$ with $i+j=\ell$, $i\le\dist_T(w,a)-1$, and
$j\le\dist_T(w,b)-1$. Let $z_a$ be the ancestor of $a$ at
distance $i$, and let $z_b$ be the ancestor of $b$ at distance
$j$. Both are strict descendants of $w$ in different branches. Let $\mathcal F_{z_a,z_b}$ denote the random tuple of all
permutations associated with vertices outside the two subtrees
rooted at $z_a$ and $z_b$. Conditioning on $\mathcal F_{z_a,z_b}$ fixes $f(z_a)$ and $f(z_b)$, while all permutations used along $T[z_a,a]$ and $T[z_b,b]$ remain independent and uniform. On $B_{a,b}$, the images of $T[z_a,a]$ and $T[b,z_b]$ meet at $f(a)=f(b)$, and their union is a simple oriented path of length $\ell$ between the fixed and distinct vertices $f(z_a)$ and $f(z_b)$.
By Lemma \ref{lem-path}, there is at most one such path.
Hence, all $i+j=\ell$ choices along these two paths are
determined. Revealing these choices successively, we obtain
\(
\mathbb P(B_{a,b}\mid\mathcal F_{z_a,z_b})
\le(d-1)^{-\ell}.
\)
Therefore, by the law of total probability,
\[
\mathbb P(B_{a,b})
=\sum_{\varphi}
\mathbb P(B_{a,b}\mid\mathcal F_{z_a,z_b}=\varphi)
\mathbb P(\mathcal F_{z_a,z_b}=\varphi)
\le(d-1)^{-\ell},
\]
where $\varphi$ ranges over all possible realizations of the
permutations in $\mathcal F_{z_a,z_b}$.

The union bound now yields
\[
 \mathbb P(f\text{ is not injective})
 \le \sum_{\{a,b\}\in\mathcal E_\ell(T,r)}\mathbb P(B_{a,b})
 \le\frac{|\mathcal E_\ell(T,r)|}{(d-1)^\ell}<1.
\]
Therefore, there is an injective map $f$ obtained by the construction, and thus $f$ is an embedding of $T$ in $D$.
\end{proof}

\Needspace{8\baselineskip}
\section{Proofs of Theorems \ref{thm-2} and \ref{thm-3}}
The two theorems use the same tree estimates, so we prove them together.

\begin{proof}[\bfseries Proofs of Theorems \ref{thm-2} and \ref{thm-3}]
Write $d=\delta^0(D)$ in Theorem \ref{thm-2}, and $d=\bar\delta^0(D)$ in Theorem \ref{thm-3}. If $\ell=1$, then $d\ge k$, and both assertions follow
immediately from the usual greedy embedding, starting with
an edge of $T$. If $d=1$, then $\Delta(T)\le1$, so $T$ is a single edge and can be embedded in $D$.

Assume now that $\ell,d\ge2$, and suppose to the contrary that
there is a nonembeddable tree in the corresponding class.
Choose a minimal nonembeddable tree $T$ with $k$ edges.  By Lemma \ref{lem-E}, $\diam(T)\ge2\ell-1$. Let $p$ be the number of penultimate vertices. Since $T$ is not a star, $p\ge2$. Lemmas \ref{lem-p1} and \ref{lem-p2} imply $p\le\ell-1$ and that each penultimate vertex has at least two leaf neighbors. Thus, $\ell,d\ge3$.

For every root $r$, every pair in $\mathcal E_\ell(T,r)$ has
distance at least $\ell+2$. Note that
$|T|\ge\diam(T)+1\ge2\ell\ge\ell+2$.
By Lemma \ref{lem-pairs}, we have
\begin{equation}\label{eq-general-count}
 |\mathcal E_\ell(T,r)|\le\binom{k-\ell}{2}
 \le\binom{\ell(d-1)}{2}.
\end{equation}

Write $q=d-1\ge2$. We first verify that
$\binom{\ell q}{2}<q^\ell$ except for six parameter pairs.
If $\ell=3$ and $q\ge5$, then
\[
q^3-\binom{3q}{2}
=\frac{q(2q^2-9q+3)}{2}>0.
\]
Suppose that $\ell\ge4$ and $q\ge3$. Since
$\binom{\ell q}{2}<\ell^2q^2/2$, it suffices to show that
$\ell^2/2<q^{\ell-2}$. For fixed $q\ge3$, the sequence
$2q^{\ell-2}/\ell^2$ is increasing for $\ell\ge4$, and its
value is at least $9/8>1$ at $\ell=4$. Hence,
$\ell^2/2<q^{\ell-2}$ for every $\ell\ge4$. Finally, suppose that $q=2$. The sequence $2^\ell/\binom{2\ell}{2}$ is increasing for $\ell\ge7$, and its value is $128/91>1$ at $\ell=7$.
Hence, $\binom{2\ell}{2}<2^\ell$ for every $\ell\ge7$. By \eqref{eq-general-count} and Lemma \ref{lem-random}, it remains to consider
\[
(\ell,d)\in
\{(3,3),(3,4),(3,5),(4,3),(5,3),(6,3)\}.
\]

For $(\ell,d)\in\{(3,5),(6,3)\}$, we have $(d-1)^\ell=64$.
Since $T$ has at least two penultimate vertices,
Lemma \ref{lem-p2} implies that $T$ has at least four leaves.
A longest path contains at most two of these leaves, so
$|T|\ge\diam(T)+1+2\ge2\ell+2\ge\ell+5$.
Applying the last assertion of Lemma \ref{lem-pairs}
with $h=\ell+2$, we find that the number of pairs at
distance at least $\ell+2$ is at most
\[
\binom{k-\ell}{2}-2
\le \binom{\ell(d-1)}{2}-2
=64.
\]
Choose an endpoint of a longest path as $r$.
Since $\diam(T)\ge2\ell-1\ge\ell+2$, this path contains
a vertex at distance $\ell+2$ from $r$.
The pair consisting of $r$ and this vertex is included in
the preceding count, but it does not belong to
$\mathcal E_\ell(T,r)$ because its vertices are comparable
and its distance is $\ell+2<2\ell+1$.
Thus, $|\mathcal E_\ell(T,r)|\le64-1=63<(d-1)^\ell$.
By Lemma \ref{lem-random}, $T$ can be embedded in $D$,
a contradiction.

We treat the four remaining parameter pairs in three cases.
The following notation will be used whenever $S(T)$ is a path
$v_0v_1\cdots v_s$. For each $i\in\{0,\ldots,s\}$, let $t_i$
be the number of leaves of $T$ adjacent to $v_i$, and choose
a leaf at $v_0$ as $r$. Denote by $F$ the number of pairs $\{x,y\}\subseteq V(T)$ satisfying
$\dist_T(x,y)\ge2\ell+1$, and by $M_1$ and $M_2$ the numbers
of incomparable leaf--skeleton pairs and leaf--leaf pairs
$\{x,y\}$ satisfying
$\ell+2\le\dist_T(x,y)\le2\ell$, respectively.
Since any two vertices of $S(T)$ are comparable, we have
\begin{equation}\label{eq-three-counts}
|\mathcal E_\ell(T,r)|=F+M_1+M_2.
\end{equation}
For $0\le i<j\le s$, let $x$ and $y$ be leaves adjacent to
$v_i$ and $v_j$, respectively. Then
$\dist_T(v_i,v_j)=j-i$, $\dist_T(x,v_j)=j-i+1$, and
$\dist_T(x,y)=j-i+2$. Hence, the pairs $\{v_i,v_j\}$,
$\{x,v_j\}$, and $\{x,y\}$ are counted by $F$ exactly when
$j-i\ge2\ell+1$, $j-i\ge2\ell$, and $j-i\ge2\ell-1$,
respectively. When $i=0$, pairs containing $r$ are excluded from the
counts for $M_1$ and $M_2$, since $r$ is comparable
with every vertex. It follows that
\begin{equation}\label{eq-middle-counts}
\begin{split}
M_1&=\sum_{\substack{0\le i<j\le s\\
\ell+1\le j-i\le2\ell-1}}
\bigl(t_i-\mathbf1_{\{i=0\}}\bigr),\\
M_2&=\sum_{\substack{0\le i<j\le s\\
\ell\le j-i\le2\ell-2}}
\bigl(t_i-\mathbf1_{\{i=0\}}\bigr)t_j.
\end{split}
\end{equation}
Here $\mathbf1_{\{i=0\}}$ is $1$ if $i=0$ and $0$ otherwise.

\smallskip
\noindent\textbf{Case 1.} $\ell=3$ and $d\in\{3,4\}$.

Since $2\le p\le\ell-1=2$, we have $p=2$, and hence
$S(T)=v_0v_1\cdots v_s$ is a path. By Lemma \ref{lem-p2},
$t_0,t_s\ge2$. For every $i\in\{1,\ldots,s-1\}$, the two
nontrivial components of $T-v_i$ containing $v_0$ and $v_s$,
respectively, each have at least three vertices.
Lemma \ref{lem-p2} therefore implies that $t_i\ne1$.

Suppose first that $d=3$. Then $t_0=t_s=2$ and
$t_i=0$ for every $i\in\{1,\ldots,s-1\}$.
Since $s+4=k\le9$, we have $s\le5$.
When $s\le4$, $\diam(T)=s+2\le6<7=2\ell+1$, and hence
$F=0$. By \eqref{eq-middle-counts}, $M_1\le1$ and
$M_2\le2$. Therefore,
$|\mathcal E_3(T,r)|\le3<8=(d-1)^3$.
When $s=5$, we have $F=4$, $M_1=2$, and $M_2=0$.
Thus, $|\mathcal E_3(T,r)|\le6<8=(d-1)^3$.

Suppose now that $d=4$. By symmetry, we may assume that
$2\le t_0\le t_s\le3$. For every
$i\in\{1,\ldots,s-1\}$, we have $t_i\in\{0,2\}$.
Let $m$ be the number of internal vertices $v_i$ with $t_i=2$. Then 
\begin{equation}\label{eq-edge}
 s+t_0+t_s+2m=k\le12   .
\end{equation}
Since $s+2=\diam(T)\ge5$, we have $3\le s\le8$ and $m\le2$.

When $3\le s\le4$, we have $F=0$ and $M_1\le2$.
If $s=4$, the index pairs in the sum defining $M_2$ are
$(0,3)$, $(0,4)$, and $(1,4)$, so
$M_2\le2(2+3)+2\cdot3=16$.
If $s=3$, then $M_2=(t_0-1)t_3\le6$.
Thus, $|\mathcal E_3(T,r)|\le18$ for $3\le s\le4$.

When $s=5$, we have $m\le1$. The pairs counted by $F$ are
exactly the $t_0t_s$ pairs consisting of one leaf adjacent to
$v_0$ and one leaf adjacent to $v_s$, so $F=t_0t_s\le9$.
If $m=0$, then $M_1=2(t_0-1)\le4$ and $M_2=0$.
If $m=1$, \eqref{eq-edge} implies $t_0=2$.
The two leaves adjacent to the unique internal vertex with
leaf neighbors account for at most two pairs in $M_1$,
and they form pairs counted by $M_2$ with leaves adjacent
to at most one of $v_0$ and $v_s$.
Hence, $M_1\le2(t_0-1)+2=4$ and
$M_2\le\max\{2t_s,2(t_0-1)\}\le6$.
Therefore, $|\mathcal E_3(T,r)|\le19$.

When $s=6$, \eqref{eq-edge} implies that $m\le1$.
If $m=0$, then $F=t_0+t_s+t_0t_s\le15$,
$M_1=2(t_0-1)\le4$, and $M_2=0$.
If $m=1$, then \eqref{eq-edge} implies that $t_0=t_s=2$.
The leaves adjacent to $v_0$ and $v_s$ account for
$t_0+t_s+t_0t_s=8$ pairs in $F$, and the two leaves adjacent
to the unique internal vertex with leaf neighbors account for
at most four more. Thus, $F\le12$.
For $M_1$, the nonroot leaf adjacent to $v_0$ accounts for
two pairs, and the two internal leaves account for at most four,
so $M_1\le6$. For $M_2$, the two internal leaves form at most
two pairs with leaves adjacent to $v_0$ and at most four with
leaves adjacent to $v_s$, so $M_2\le6$.
Therefore, $|\mathcal E_3(T,r)|\le24$.

When $s=7$, \eqref{eq-edge} implies $m=0$, $t_0=2$, and
$t_s\le3$. Thus, every leaf is adjacent to $v_0$ or $v_s$.
The numbers of skeleton--skeleton, leaf--skeleton, and
leaf--leaf pairs at distance at least $7$ are $1$,
$2(t_0+t_s)$, and $t_0t_s$, respectively.
Thus, $F=1+2(t_0+t_s)+t_0t_s\le17$, while
\eqref{eq-middle-counts} implies $M_1=2$ and $M_2=0$.
Therefore, $|\mathcal E_3(T,r)|\le19$.

When $s=8$, \eqref{eq-edge} implies $m=0$ and
$t_0=t_s=2$. There are three pairs of skeleton vertices,
twelve leaf--skeleton pairs, and four leaf--leaf pairs at
distance at least $7$, so $F=19$.
By \eqref{eq-middle-counts}, $M_1=2$ and $M_2=0$.
Therefore, $|\mathcal E_3(T,r)|=21$.

Combining the above bounds, we obtain
$|\mathcal E_3(T,r)|\le24<27=(d-1)^3$ for $d=4$.
By Lemma \ref{lem-random}, $T$ can be embedded in $D$,
a contradiction.

\smallskip
\noindent\textbf{Case 2.} $\ell=4$ and $d=3$.

Since $2\le p\le\ell-1=3$, $S(T)$ has two or three leaves.
Suppose first that $S(T)=v_0v_1\cdots v_s$ is a path.
Then $t_0=t_s=2$ and $t_i\in\{0,1\}$ for every
$i\in\{1,\ldots,s-1\}$. If $t_i=1$, the two nontrivial
components of $T-v_i$ containing $v_0$ and $v_s$ have at least
$i+2$ and $s-i+2$ vertices, respectively. By
Lemma \ref{lem-p2}, one of them has fewer than four vertices,
and hence $i=1$ or $i=s-1$. Let $m$ be the number of internal
vertices $v_i$ with $t_i=1$. Then $m\le2$,
$s+4+m=k\le12$, and $s+2=\diam(T)\ge7$.

When $m=0$, we have $s\le8$. No pair of skeleton vertices has
distance at least $9$, and there are at most four
leaf--skeleton pairs and four leaf--leaf pairs at distance at
least $9$. Thus, $F\le8$. By \eqref{eq-middle-counts},
$M_1\le3$ and $M_2\le2$. Therefore,
$|\mathcal E_4(T,r)|\le13$.

When $m=1$, we have $s\le7$. A pair at distance at least $9$
must consist of leaves adjacent to $v_0$ and $v_s$, so
$F\le4$. The nonroot leaf adjacent to $v_0$ contributes at
most three pairs to $M_1$. The leaf adjacent to an internal
vertex contributes at most two if it is adjacent to $v_1$,
and none if it is adjacent to $v_{s-1}$. Hence, $M_1\le5$.
For $M_2$, the leaves adjacent to $v_0$ and $v_s$ contribute
at most two pairs, and the leaf adjacent to an internal vertex
contributes at most two pairs with leaves at the opposite end.
Thus, $M_2\le4$ and $|\mathcal E_4(T,r)|\le13$.

When $m=2$, we have $s\le6$, and the two leaves adjacent to
internal vertices are adjacent to $v_1$ and $v_{s-1}$.
Since $\diam(T)=s+2\le8$, we have $F=0$.
The nonroot leaf adjacent to $v_0$ contributes at most two
pairs to $M_1$, and the leaf adjacent to $v_1$ contributes
at most one, so $M_1\le3$. For $M_2$, there are at most two
pairs between leaves at the two ends, one pair between the
nonroot leaf at $v_0$ and the leaf at $v_{s-1}$, two pairs
between the leaf at $v_1$ and the leaves at $v_s$, and one
pair between the leaves at $v_1$ and $v_{s-1}$. Hence,
$M_2\le6$ and $|\mathcal E_4(T,r)|\le9$.
Therefore, $|\mathcal E_4(T,r)|\le13<16$ when $S(T)$ is
a path.

Suppose now that $S(T)$ has three leaves. By
Lemma \ref{lem-2} and $\Delta(T)\le3$, it consists of three
paths of lengths $a\le b\le c$ from a common vertex.
Let $m$ be the number of leaves adjacent to internal vertices
of these paths. Each penultimate vertex is adjacent to two
leaves, so $a+b+c+6+m\le12$. Moreover,
$b+c+2=\diam(T)\ge7$. Since $a\ge1$, we have $m=0$, $a=1$,
and $b+c=5$. Thus,
$(a,b,c)=(1,1,4)$ or $(1,2,3)$.

Choose as $r$ a leaf adjacent to the end of the longest path.
There is no pair at distance at least $9$, and any two vertices
in the two shorter branches have distance at most $5$.
Hence, every pair in $\mathcal E_4(T,r)$ contains the sibling
of $r$. For $(a,b,c)=(1,1,4)$, this sibling forms two pairs
with the endpoints of the shorter paths and four pairs with
their leaves. For $(a,b,c)=(1,2,3)$, it forms two pairs with
the leaves at the end of the shortest path and three pairs
with the endpoint and leaves at the end of the other path.
Thus, $|\mathcal E_4(T,r)|\le6$.

Combining the two possibilities for $S(T)$, there is a root
$r$ such that
$|\mathcal E_4(T,r)|\le13<16=(d-1)^4$.
By Lemma \ref{lem-random}, $T$ can be embedded in $D$,
a contradiction.

\smallskip
\noindent\textbf{Case 3.} $\ell=5$ and $d=3$.

Since $2\le p\le\ell-1=4$, $S(T)$ has two, three, or four
leaves. Suppose first that $S(T)=v_0v_1\cdots v_s$ is a path.
Then $t_0=t_s=2$ and $t_i\in\{0,1\}$ for every
$i\in\{1,\ldots,s-1\}$. If $t_i=1$, Lemma \ref{lem-p2}
implies that one of the two nontrivial components of $T-v_i$
has fewer than five vertices. Since these components have at
least $i+2$ and $s-i+2$ vertices, respectively, we have
$i\in\{1,2,s-2,s-1\}$. Moreover, $t_1$ and $t_2$ cannot both
be $1$. Otherwise, both nontrivial components of $T-v_2$ have
at least five vertices, contrary to Lemma \ref{lem-p2}.
Similarly, $t_{s-2}$ and $t_{s-1}$ cannot both be $1$.
Let $m$ be the number of internal vertices $v_i$ with $t_i=1$.
Then $m\le2$, $s+4+m=k\le15$, and $s\ge7$.

When $m=0$, we have $s\le11$. At most one
skeleton--skeleton pair, eight leaf--skeleton pairs, and four
leaf--leaf pairs have distance at least $11$, so $F\le13$.
By \eqref{eq-middle-counts}, $M_1\le4$ and $M_2\le2$.
Therefore, $|\mathcal E_5(T,r)|\le19$.

When $m=1$, we have $s\le10$. There are no
skeleton--skeleton pairs at distance at least $11$.
The leaves adjacent to $v_0$ and $v_s$ account for at most
four leaf--skeleton pairs and four leaf--leaf pairs, while the
unique leaf adjacent to an internal vertex accounts for at most
two further pairs. Thus, $F\le10$.
The nonroot leaf adjacent to $v_0$ and the leaf adjacent to an
internal vertex each contribute at most four pairs to $M_1$,
so $M_1\le8$. The leaves at the two ends contribute at most
two pairs to $M_2$, and the leaf adjacent to an internal vertex
contributes at most two more. Hence, $M_2\le4$ and
$|\mathcal E_5(T,r)|\le22$.

When $m=2$, we have $s\le9$. The pairs counted by $F$ consist
of one leaf adjacent to $v_0$ and one leaf adjacent to $v_s$,
so $F\le4$. The nonroot leaf adjacent to $v_0$ contributes
at most four pairs to $M_1$, the leaf adjacent to an internal
vertex near $v_0$ at most three, and the leaf near $v_s$ none.
Thus, $M_1\le7$. For $M_2$, there are at most two pairs between
leaves at the two ends, one pair between the nonroot leaf at
$v_0$ and the internal leaf near $v_s$, two pairs between the
internal leaf near $v_0$ and the leaves at $v_s$, and one pair
between the two internal leaves. Hence, $M_2\le6$ and
$|\mathcal E_5(T,r)|\le17$.

Combining the three subcases, we have
$|\mathcal E_5(T,r)|\le22<32=(d-1)^5$ when $S(T)$ is a path.
By Lemma \ref{lem-random}, $T$ can be embedded in $D$,
a contradiction.

Suppose now that $S(T)$ has three leaves. By
Lemma \ref{lem-2} and $\Delta(T)\le3$, it consists of three
paths of lengths $a\le b\le c$ from a common vertex.
Let $m$ be the number of leaves adjacent to internal vertices
of these paths. Each penultimate vertex is adjacent to two
leaves, so $a+b+c+6+m\le15$. Moreover,
$b+c+2=\diam(T)\ge9$. Hence, $a+b+c\ge8$ and $m\le1$.

Delete the leaf adjacent to an internal vertex if it exists,
and denote the resulting tree by $T_0$. We call the three components obtained by deleting the common vertex from $T_0$ its branches. Choose as $r$ a leaf whose neighbor is at distance $c$
from the common vertex. Since $a+b+c\le9$, the diameter of $T_0$ is at most $b+c+2\le10$, so there are no pairs at distance at least $11$.
Every pair in $\mathcal E_5(T_0,r)$ either contains the sibling
of $r$ or has one vertex in each of the two shorter branches.
By the estimate for a leaf with a sibling in the proof of
Lemma \ref{lem-pairs}, at most $|T_0|-7-1\le8$ vertices are
at distance at least $7$ from the sibling of $r$.
Thus, at most eight pairs in $\mathcal E_5(T_0,r)$
contain this sibling.
For pairs with one vertex in each of the two shorter
branches, note that $b+c\ge7$ implies $c\ge4$, and hence
$a+b\le5$. Unless both vertices are leaves of $T_0$, their distance
is at most $a+b+1\le6$.
Therefore, at most four such pairs belong to
$\mathcal E_5(T_0,r)$, and
$|\mathcal E_5(T_0,r)|\le8+4=12$.
If $m=0$, then $T_0=T$. If $m=1$, then $T_0$ has at most
$15$ vertices. Adding the deleted leaf changes neither the
distances nor the comparability of the vertices of $T_0$.
The new leaf is at distance at most two from its parent and
the two skeleton neighbors of its parent, so it belongs to
at most $15-3=12$ new pairs. Therefore, $|\mathcal E_5(T,r)|\le24<32=(d-1)^5$ for both $m=0$ and $m=1$.
By Lemma \ref{lem-random}, $T$ can be embedded in $D$, a contradiction.

Suppose finally that $S(T)$ has four leaves.
By Lemma \ref{lem-p2}, these four penultimate vertices are
each adjacent to at least two leaves, so $T$ has at least
eight leaves. Since $|T|\le16$, we have $|S(T)|\le8$.
Note that an internal vertex of a path in $S(T)$ cannot be a leaf
of $S(T)$. We see that every path in $S(T)$ has at most $8-4=4$
internal vertices and length at most five.
Consequently, $\diam(T)=\diam(S(T))+2\le7$, contrary to
$\diam(T)\ge9$. This completes the proof.
\end{proof}

\section*{Declarations}
The authors declare that they have no known competing financial interests or personal relationships that could have appeared to influence the work reported in this paper.
\section*{Availability of Data and Materials}  
Not applicable.
\section*{Acknowledgments}
This research was supported by National Key R\&D Program of China under grant number 2024YFA1013900, NSFC under grant number 12471327.


{\small
\begin{thebibliography}{xx}
\small \setlength{\itemsep}{-.10mm}
\bibitem{Brandt}S. Brandt, E. Dobson, The Erd\H{o}s-S\'{o}s conjecture for graphs of girth $5$, Discrete Math. 150 (1996) 411--414.


\bibitem{Dobson}
E. Dobson,
Some problems in algebraic and extremal graph theory,
Ph.D. dissertation,
Louisiana State University, Baton Rouge, 1995.

\bibitem{Erdos}
P. Erd\H{o}s,
Extremal problems in graph theory,
in M. Fiedler (Ed.),
Theory of Graphs and its Applications,
Publishing House of the Czechoslovak Academy of Sciences,
Prague, 1964, pp. 29--36.

\bibitem{Sachs}P. Erd\H{o}s, H. Sachs, Regulare Graphen gegebener Taillenweite mit minimaler Knotenzahl, Wiss. Z. Uni. Halle (Math. Nat.), 12 (1963) 251--257.

\bibitem{Haxell}P. E. Haxell, T. \L uczak, Embedding trees into graphs of large girth, Discrete Math. 216
(2000) 273--278.

\bibitem{Jiang} T. Jiang, On a conjecture about trees in graphs with large girth, J. Combin. Theory Ser. B 83 (2001) 221--232.

\bibitem{Sacle} J.-F. Sacl\'{e}, M. Wo\'{z}niak, The Erd\H{o}s--S\'{o}s conjecture for graphs without $C_4$, J. Combin. Theory Ser. B, 70 (1997) 367--372.

\bibitem{Stein}
M. Stein, A. Trujillo-Negrete,
Oriented trees in digraphs without oriented $4$-cycles,
Discrete Math. 349 (2026) 115288.

\end{thebibliography}
}
\end{document}